\documentclass[11pt]{article} 
\usepackage[fleqn]{amsmath}                  
\usepackage{amssymb}                  
\usepackage{amsthm}                   
\usepackage[]{geometry}                 
\usepackage{graphicx}                 
\usepackage{float}                    
\usepackage{caption}                  
\usepackage{cite}
\usepackage{blkarray}
\usepackage{booktabs}   
\usepackage{array}      
\usepackage{multirow}   
\usepackage{breqn}
\usepackage{subfigure}
\usepackage[affil-it]{authblk}
\usepackage[english]{babel}
\usepackage{blindtext}
\usepackage[symbol]{footmisc}
\def\correspondingauthor{\footnote{Corresponding author.E-mail addresses: umesh@maths.iitkgp.ac.in
(U.C.Gupta),nitinkumar7276@gmail.com (Nitin
Kumar),faridaparvezb@gmail.com (F. P. Barbhuiya)}}

\newtheorem{theorem}{Theorem}
\theoremstyle{definition}

\begin{document}
\title{A queueing system with batch renewal input and negative arrivals}
%
%

%
\author{U. C. Gupta\correspondingauthor{},~Nitin Kumar, ~F. P. Barbhuiya
\\{\it Department of Mathematics, Indian Institute of Technology Kharagpur,}\\
{\it Kharagpur-721302, India}}

\maketitle{} \textbf{Abstract:} This paper studies an infinite
buffer single server queueing model with exponentially distributed
service times and negative arrivals. The ordinary (positive)
customers arrive in batches of random size according to renewal
arrival process, and joins the queue/server for service. The
negative arrivals are characterized by two independent Poisson
arrival processes, a negative customer which removes the positive
customer undergoing service, if any, and a disaster which makes the
system empty by simultaneously removing all the positive customers
present in the system. Using the supplementary variable technique
and difference equation method we obtain explicit formulae for the
steady-state distribution of the number of positive customers in the
system at pre-arrival and arbitrary epochs. Moreover, we discuss the
results of some special models with or without negative arrivals
along with their stability conditions. The results obtained
throughout the analysis are computationally tractable as illustrated
by few numerical examples. Furthermore, we discuss the impact of the
negative arrivals on the performance of the system by means of some
graphical
representations.\\
\textbf{Keywords:} Batch arrival, Difference equation, Disasters,
RCH, Renewal process, Negative customers
\section{Introduction}\label{sec1}
Since the pioneering work of Gelenbe \cite{gelenbe1989random} in the
year 1989, queueing models with negative arrivals (also termed as
G-networks) have gained considerable attention. A negative arrival
causes the removal of one or more ordinary customer (also called
positive customer) from the system, and prevents it from getting
served. In the literature, negative arrivals are generally
introduced by the name of `negative customers' and/or `disasters'.
The arrival of a negative customer removes one ordinary customer
from the system, according to a definite killing strategy i.e., RCH
(Removal of Customer at the Head) or RCE (Removal of Customer at the
End). Under RCH killing discipline, the customer who is undergoing
service gets removed, while in case of RCE, the customer at the end
of the queue is eliminated. Meanwhile, the occurrence of a disaster
simultaneously removes all the present customers in the system thus
making the system idle. Disasters are also known by the terms
catastrophic events (Barbhuiya et al. \cite{barbhuiya2019batch}),
mass exodus (Chen and Renshaw \cite{chen1997m}) or queue flushing
(Towsley and Tripathi \cite{towsley1991single}). Both, a negative
customer and a disaster have no impact on the system when it is
empty. For further references on different queueing models with
negative arrivals the readers may refer to the bibliography by Van
Do \cite{van2011bibliography}.
\par Initially, $M/M/1$ queueing model with positive and negative
customers was studied by Harrison and Pitel
\cite{harrison1993sojourn}. They derived the Laplace transforms of
the sojourn time density under both RCH and RCE killing discipline.
They further extended their work to $M/G/1$ queue with negative
arrivals and obtained the generating function of the queue length
probability distribution (see Harrison and Pitel
\cite{harrison1995m,harrison1996m}). Jain and Sigman
\cite{jain1996pollaczek} derived a Pollaczek-Khintchine formula for
an $M/G/1$ queue with disasters using preemptive LIFO discipline,
whereas Boxma et al. \cite{boxma2001clearing} considered the same
model by assuming the disasters to occur in deterministic
equidistant times or at random times. The $M/M/1$ queue with
negative arrivals was first extended to the $GI/M/1$ queue by Yang
and Chae \cite{yang2001note}, assuming the occurrence of negative
customers (under RCE killing discipline) and disasters. Meanwhile,
Abbas and A{\"\i}ssani \cite{abbas2010strong} investigated the
strong stability conditions of the embedded Markov chain for
$GI/M/1$ queue with negative customers. A discrete-time $GI/G/1$
queue with negative arrivals was considered by Zhou
\cite{zhou2005performance} where he derived the probability
generating function of actual service time of ordinary customers.
Recently, Chakravarthy \cite{chakravarthy2017catastrophic}
investigated a single server catastrophic queueing model assuming
the arrival process to be versatile Markovian point process with
phase type service time. All the work discussed till now was studied
under steady-state condition.  Kumar and Arivudainambi
\cite{kumar2000transient} and Kumar and Madheswari
\cite{kumar2002transient} obtained the transient solution of system
size for the $M/M/1$ and $M/M/2$ queueing model with catastrophes,
respectively. Following this, a time dependent solution for the
system size of $M/M/c$ queue with heterogeneous servers and
catastrophes was considered by Dharmaraja and Kumar
\cite{dharmaraja2015transient}. A survey on queueing models with
interruptions due to various reasons such as catastrophes, server
breakdowns, etc. can be found in Krishnamoorthy et al. \cite{krishnamoorthy2014queues}.
\par The papers referred above, studies
queueing models with negative arrivals of one form or the other,
under the assumption of single arrival of positive customers. But in
most of the real-world scenario, the request for service arrives in
groups of random size. For example, transmission of messages to the
service station occurs in the form of packets in batches, unfinished
goods arrives in bulk into the production systems for further
processing. This gives us a practical motivation to relax the
assumption of single arrival and consider batch arrival of the
positive customers into the system. We study a continuous-time
$GI^X/M/1$ queue which is influenced by negative customers (with RCH
killing discipline) and disasters, occurring independently of one
another according to Poisson process. The arrival of negative
customers or disasters have no impact on the system when it is
empty. We first formulate the model using the supplementary variable
technique and then apply difference equation method to obtain the
steady-state distribution of the number of positive customers in the
system at different epochs. In the literature, most of the queueing
models with negative arrivals are studied using the matrix geometric
(matrix analytic) method or the embedded Markov chain technique.
However, encouraged by some recent works (see Barbhuiya and Gupta
\cite{barbhuiya2019difference,barbhuiya2019discrete}, Goswami and
Mund \cite{goswami2011analysis}), we try to implement the
methodology based on supplementary variable technique and difference
equation method to study queueing models with negative arrivals. The
whole procedure involved is analytically tractable and easy to
implement, as we obtain explicit formulae of the system content
distribution at pre-arrival and arbitrary epochs simultaneously, in
terms of roots of the associated characteristic equation and the
corresponding constants. We discuss the stability conditions along
with some special cases of the model. We also present some numerical
results in order to illustrate the applicability of our theoretical
work and study the influence of different parameters on the system
performance.
%
%
%
\par The queueing model described above may have possible use
in computer communications and manufacturing systems (see Artalejo
\cite{artalejo2000g}). A real-world application can be experienced
within a network of computers, where a message affected with virus
often infects the whole system when it gets transferred from one
node to another. A signal which immediately removes the message and
prevents further transmission of it can be thought of as a negative
customer. Moreover, a reset instruction in the computer database may
be considered as a disaster as it clears all the stored files present in the system. In these
systems, the stored files/data act as positive customers whereas
clearing operation plays the role of the negative arrivals (see Wang et al. \cite{wang2011discrete}, Atencia and Moreno
\cite{atencia2004discrete}).

\par The remaining portion of the paper is organized as follows. In
Section \ref{sec2} we give a comprehensive description of the model
under consideration. In Section \ref{sec3} we perform the
steady-state analysis of the model and discuss the stability
condition. We deduce the results of some special cases of our model
in Section \ref{sec4} which is followed by some illustrative
numerical examples in Section \ref{sec5}. Finally, we give the
concluding remarks in Section \ref{sec6}.
\section{Model description}\label{sec2}
\begin{figure}
\begin{center}
\includegraphics[height= 6.8 cm, width = 15 cm]{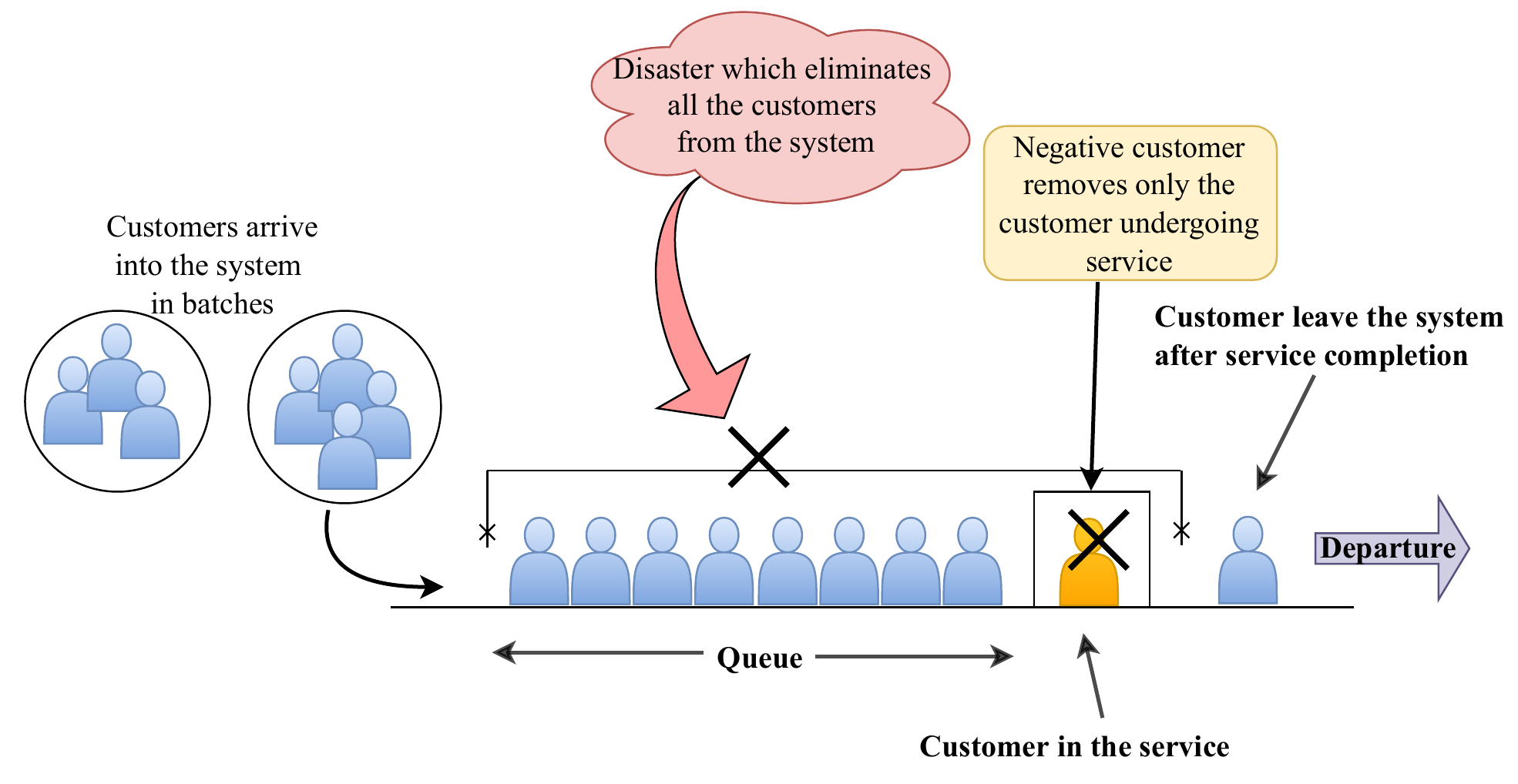}
\caption{Pictorial representation of the $GI^X/M/1$ queue with
negative customer and disaster}\label{fig_model}
\end{center}
\end{figure}
We consider an infinite buffer queueing model wherein customers
(positive customers) arrive into the system in batches and joins the
queue. The arriving batch size is a random variable $X$ with
probability mass function $P(X=i)=g_i$, $i=1,2, \ldots$. For
theoretical analysis and numerical implementation we assume that the
maximum permissible size of the arriving batch is $b$, which also
holds true in many real-world circumstances. Consequently, the mean
arriving batch size is $\overline{g}= \sum_{i=1}^big_i$ and the
probability generating function is $G(z)=\sum_{i=1}^bg_iz^i$. The
inter-arrival times $T$ between the batches are independent and
identically distributed continuous random variables with probability
density function (pdf) $a(t)$, distribution function $A(t)$, the
Laplace-Stieltjes transform (L.S.T) $A^*(s)$ and the mean
inter-arrival time $\lambda^{-1} = a = -A^{*(1)}(0)$, where
$\lambda$ is the arrival rate of the batches and $A^{*(1)}(0)$ is
the derivative of $A^*(s)$ evaluated at $s = 0$. The customers are
served individually by a single server and the service time follows
exponential distribution with parameter $\mu$.
\par The system is affected by negative arrivals which is
characterized by two independent Poisson arrival processes namely,
negative customers and disasters with rate $\eta$ and $\delta$,
respectively. The negative customer follows RCH killing discipline
and removes only the customer undergoing service, while the
occurrence of a disaster eliminates all the customers from the
system. We further assume that the negative customer or disaster
have no impact on the system when it is empty. The arrival process,
service process and the negative arrivals are independent of each
other. The model described above may be mathematically denoted by
$GI^X/M/1$ queue with negative customers and disasters. One may
refer to Figure \ref{fig_model} for a pictorial representation of
the model.

\section{The steady-state analysis} \label{sec3}
In this section we analyze the model described in Section \ref{sec2}
in steady-state. We first formulate the governing equations of the
system using supplementary variable technique (SVT) by considering
the remaining inter-arrival time of the next batch as the
supplementary variable. For this purpose, we denote the states of
the system $N(t)$ and $U(t)$ respectively, as the number of
customers in the system and the remaining inter-arrival time of the
next batch, at time $t$. We further define
$$q_n(u,t)du=P[N(t)=n, u < U(t)\leq u+ du], ~~n \geq 0, u\geq 0,$$
and in steady-state $$p_n(u)=\lim_{t \rightarrow \infty}q_n(u,t).$$
Relating the states of the system at two consecutive epochs $t$ and
$t+\Delta t$ and using the arguments of SVT we obtain the following
difference-differential equations in steady state.
%
\begin{eqnarray}
-\frac{d}{du}p_0(u) &=&  (\mu+\eta)  p_1(u) + \delta \sum_{k=1}^{\infty} p_k(u),  \label{3}\\
-\frac{d}{du}p_n(u) &=&  -(\mu+\eta+\delta) p_n(u)  + a(u)
\sum_{i=1}^{min\{n,b\}} g_i p_{n-i}(0) + (\mu+\eta) p_{n+1}(u), ~~n
\ge 1. \label{4}
\end{eqnarray}
Obtaining the steady-state solution directly from (\ref{3}) and
(\ref{4}) is a rather difficult task. Therefore, for further
analysis we take the transform for which we define $$p_n^*(s) =
\int_{0}^{\infty} e^{-su} p_n(u) du ~\Rightarrow~
p_{n}=p_{n}^{*}(0)=\int_{0}^{\infty}p_{n}(u)du,~ n\geq 0.$$
Multiplying (\ref{3})-(\ref{4}) by $e^{-su}$, integrating with
respect to $u$ over $0$ to $\infty$ and then separating equation
(\ref{4}) we obtain the transformed equations as
\begin{eqnarray}
-  s p^*_0(s) &=& (\mu+\eta) p^*_1(s)  + \delta \sum_{k=1}^{\infty} p_k^*(s) - p_0(0), \label{5} \\
(\mu +\eta+ \delta - s) p^*_n(s) &=& A^*(s) \sum_{i=1}^{n} g_i p_{n-i}(0) + (\mu+\eta) p^*_{n+1}(s) - p_n(0), ~ 1 \le n \leq b-1, \label{6} \\
(\mu +\eta+ \delta - s) p^*_n(s) &=& A^*(s) \sum_{i=1}^{b} g_i
p_{n-i}(0) + (\mu+\eta) p^*_{n+1}(s) - p_n(0), ~~ n \ge b. \label{7}
\end{eqnarray}
Adding (\ref{5})-(\ref{7}) for all values of $n$, taking limit $s
\rightarrow 1$ and using the normalizing condition
$\sum_{n=0}^{\infty}p_n=1$ we have
\begin{eqnarray}
\sum\limits_{n=0}^{\infty} p_n(0)= \frac{1}{a}= \lambda. \label{8}
\end{eqnarray}
The L.H.S of equation (\ref{8}) denotes the mean number of arriving
batch into the system per unit time such that the remaining
inter-arrival time is 0, which is actually the arrival rate
$\lambda$. We now define $p_n^-$ as the probability that the number
of positive customers in the system is $n$ just before the arrival
of a batch, i.e., at pre-arrival epoch. Since $p_n^-$ is
proportional to $p_n(0)$ and $\sum_{n=0}^{\infty}p_n^-=1$, we have
the relation between $p_n^-$ and $p_n(0)$ as
\begin{eqnarray}
p_n^-=\frac{p_n(0)}{\sum_{k=0}^{\infty} p_k(0)}=
\frac{p_n(0)}{\lambda},~~ n \ge 0.\label{9}
\end{eqnarray}
Based on the theory of difference equations we obtain the state
probabilities at pre-arrival ($p_n^-$) and arbitrary ($p_n$) epochs
in the following section.
\subsection{Steady-state system-content
distributions}\label{subsec3.1}
We define the right shift operator $D$ on the sequence of
probabilities $\{p_n(0)\}$ and $\{p_n^*(s)\}$ as $Dp_n(0) =
p_{n+1}(0)$ and $Dp_n^*(s) = p_{n+1}^*(s)$ for all $n$. Thus,
(\ref{7}) can be rewritten in the form
\begin{eqnarray}
\left(\delta - s + (\mu+\eta)(1- D) \right) p^*_n(s) &=&   \left(
A^*(s) \sum_{i=1}^{b} g_i D^{b-i}   -D^b \right) p_{n-b}(0), ~~ n
\ge b. \label{10}
\end{eqnarray}
Substituting $s = \delta + (\mu+\eta) (1-D)$ in (\ref{10}), we get
the following homogeneous difference equation with constant
coefficient:
\begin{eqnarray}
\left[A^*(\delta + (\mu+\eta) (1-D)) \sum_{i=1}^{b} g_i D^{b-i} -D^b
\right] p_{n}(0) = 0, ~~ n \ge 0. \label{11}
\end{eqnarray}
The corresponding characteristic equation (c.e.) is
\begin{eqnarray}
A^*( \delta + (\mu+\eta) (1-z)) \sum_{i=1}^{b} g_i z^{b-i} - z^b =
0, \label{12}
\end{eqnarray}
which has exactly $b$ roots, denoted by $r_1, r_2, . . .,r_b$,
inside the unit circle $|z|=1$. Thus the solution of (\ref{11}) is
of the form
\begin{eqnarray}
p_{n}(0)= \sum_{i=1}^{b}c_i r_i^n, ~~n \ge 0, \label{13}
\end{eqnarray}
where $c_1, c_2, . . ., c_b$ are the corresponding $b$ arbitrary
constants independent of $n$. Now using (\ref{13}) in (\ref{10}) we
have the following non-homogeneous difference equation
\begin{eqnarray}
(\delta -s + (\mu+\eta) (1-D)) p^*_n(s) &=& \sum_{j=1}^{b}c_j \left(
A^*(s) \sum_{i=1}^{b} g_i r_j^{-i} - 1 \right) r_j^{n}, ~~ n \ge b.
\label{14}
\end{eqnarray}
The general solution of (\ref{14}) is of the form
\begin{eqnarray}
p_n^*(s)= B \left( 1+\frac{ \delta - s}{\mu+\eta} \right)^{n} +
\sum_{j=1}^{b}c_j \left\{ \frac{A^*(s) G(r_j^{-1}) - 1}{\delta - s +
(\mu+\eta)(1-r_j)} \right\} r_j^{n}, ~~ n \geq b,\label{15}
\end{eqnarray}
where the first term in the R.H.S of (\ref{15}) is the solution
corresponding to the homogeneous equation of (\ref{14}) for a fixed
$s$, such that $B$ is an arbitrary constant. Meanwhile, the second
term in the R.H.S. is a particular solution of (\ref{14}). Taking
limit as $s \rightarrow 0$ and summing over $n$ from $b$ to $\infty$
in (\ref{15}), we have,
$\sum_{n=b}^{\infty}p_{n}^{*}(0)=\sum_{n=b}^{\infty}p_{n} \leq 1$.
However, $B\sum_{n=b}^{\infty}\left( 1+\frac{ \delta - s}{\mu+\eta}
\right)^{n}$ tends to infinity as $s \rightarrow 0$. Thus to ensure
the convergence of the solution we must have $B=0$ and thus
(\ref{15}) reduces to
\begin{eqnarray}
p_n^*(s)= \sum_{j=1}^{b}c_j \left\{ \frac{A^*(s) G(r_j^{-1}) -
1}{\delta - s + (\mu+\eta)(1-r_j)} \right\} r_j^{n}, ~~ n \geq
b.\label{16}
\end{eqnarray}
We now find the conditions under which $p_n^*(s)$ satisfies
(\ref{16}) for $1 \leq n \leq b-1$ as well. Thus substituting the
respective values in (\ref{6}) we obtain
\begin{eqnarray}
\sum_{j=1}^{b}c_j \left(\sum\limits_{i=n+1}^{b} g_i r_j^{n-i}
\right) =0, ~ 1 \leq n \leq b-1, \nonumber 
\end{eqnarray}
which reduces to the following on using the condition $g_b \neq 0$,
\begin{eqnarray}
\sum_{j=1}^{b}c_j  r_j^{n-b} =0, ~ 1 \leq n \leq b-1. \label{18}
\end{eqnarray}
Summing over $n$ from $0$ to $\infty$ in (\ref{13}) and using
relation (\ref{8}) we obtain
\begin{eqnarray}
\lambda = \sum_{i=1}^{b}\frac{c_i}{1-r_{i}}. \label{19}
\end{eqnarray}
One may note that (\ref{18}) and (\ref{19}) together constitutes a
system of $b$ equations in $b$ unknowns which can be solved to
obtain the constants $c_j$ for $j=1,2,\ldots,b$. Once $c_j$'s are
obtained, the expression of $p_n(0)$ given in (\ref{13}) becomes
completely known and $p_n^*(s)$ is given by
\begin{eqnarray}
p_n^*(s)= \sum_{j=1}^{b}c_j \left\{ \frac{A^*(s) G(r_j^{-1}) -
1}{\delta - s + (\mu+\eta)(1-r_j)} \right\} r_j^{n}, ~~ n \geq
1.\label{20}
\end{eqnarray}
Now, using (\ref{9}) and (\ref{20}), the steady-state distribution
of the number of positive customers in the system at pre-arrival and
arbitrary epochs are given by
\begin{eqnarray}
p_n^-&=&\frac{1}{\lambda}\sum\limits_{i=1}^{b} c_i r_i^n, ~~n \geq 0, \label{21}\\
p_n&=&p_n^*(0)=\sum_{j=1}^{b}c_j   \left\{ \frac{G(r_j^{-1}) - 1 }{\delta + (\mu+\eta)(1-r_j)} \right\} r_j^{n}, ~ n \geq 1, \label{22}\\
p_0&=& 1- \sum_{n=1}^{\infty}p_n= 1- \sum_{j=1}^{b}  \frac{c_j
r_j}{1-r_j} \left\{ \frac{ G(r_j^{-1})- 1 }{\delta +
(\mu+\eta)(1-r_j)} \right\}. \label{23}
\end{eqnarray}
This completes the analysis of the model under consideration. It may
be noted that the results derived so far are mainly expressed in
terms of the roots of the c.e. (\ref{12}) lying inside the unit
circle. It can be proved that $\delta >0$ is a sufficient condition
for the c.e. to have exactly $b$ roots inside the unit circle (see
Appendix), which ensures the stability of the system. Or in other
words, due to the occurrence of disasters the system becomes empty
and as a result the model under consideration always remains stable.
\par Once the probability distributions are completely known,
different characteristic measures determining the performance of the
system can be easily established. For example, the average
population size at pre-arrival ($L^-$) and arbitrary ($L$) epochs
are given by $L^{-}=\sum_{n=1}^{\infty}np_n^-$ and
$L=\sum_{n=1}^{\infty}np_n$, respectively. That is,
$$L^{-}= \frac{1}{\lambda}\sum\limits_{i=1}^{b}  \frac{c_i
r_i}{(1-r_i)^2}, ~~L = \sum_{j=1}^{b} \frac{c_j r_j}{(1-r_j)^2}
\left\{ \frac{G(r_j^{-1}) - 1}{\delta + (\mu+\eta)(1-r_j)}
\right\}.$$
\section{Special cases}\label{sec4}
In this section we discuss a few special cases of the model by
considering some fixed values of the parameters. As a result our
model reduces to some well-known classical queueing models with or
without negative arrivals.
\begin{description}
\item[Case 1:] If $\eta =0$ and $\delta =0$, i.e., negative customer
or disaster does not occur or their occurrence have no impact on the
system, then our model reduces to the classical $GI^X/M/1$ queue.
Consequently, the steady-state distributions of the number of
customers in the system at pre-arrival and arbitrary epochs can be
obtained directly from (\ref{21})-(\ref{23}) by putting $\eta =0$
and $\delta =0$, where $r_j$, $j=1,2,\ldots , b$ are the roots of
the c.e. $z^b-A^*(\mu - \mu z)\sum_{i=1}^bg_iz^{b-i}=0$ lying inside
the unit circle, and then the corresponding arbitrary constants
$c_j$, $j=1,2,\ldots , b$ can be obtained by solving the system of
equations (\ref{18}) and (\ref{19}). Here it may be noted that
$\lambda \overline{g}<\mu$ is the necessary and sufficient condition
for the stability of the system. This particular queueing model has
been extensively studied in the literature, both analytically and
numerically, based on the use of embedded Markov chain technique and
roots method (see Chaudhry and Templeton \cite{chaudhry1983first},
Bri$\grave{e}$re and Chaudhry \cite{briere1987computational}, Easton
et al. \cite{easton1984some,easton1984some1}). However, the present
paper provides an alternative procedure for the solution of the
model which is theoretically tractable and easy to implement, as
compared to the other approaches.
\par Meanwhile, setting $\eta =0$, $\delta =0$, $g_1=1$ and $g_i=0$ for
$i \geq 2$ will give the steady-state solution for $GI/M/1$ queue.
The c.e. will have a single root inside the unit circle (say $r$)
under the condition $\lambda< \mu$, and the corresponding arbitrary
constant can be obtained from (\ref{19}) as $c=\lambda (1-r)$. It is
followed by the system-content distributions which can be obtained
from (\ref{21})-(\ref{23}).
\item[Case 2:] If $\delta =0$, i.e., the disaster does not play any
role and the only negative arrivals are the negative customers, then
the model reduces to $GI^X/M/1$ queue with negative customers. The
c.e. $z^b-A^*((\mu + \eta) - (\mu +\eta) z)\sum_{i=1}^bg_iz^{b-i}=0$
will have exactly $b$ roots inside the unit circle under the
necessary and sufficient condition $\lambda \overline{g}<\mu +
\eta$. Equations (\ref{18}) and (\ref{19}) can be solved for the
arbitrary constants following which, the steady-state distributions
of the number of positive customers in the system can be obtained
from (\ref{21})-(\ref{23}). As discussed in Case 1, the solution for
$GI/M/1$ queue with negative customers (Yang and Chae
\cite{yang2001note}) can be further derived by assuming $g_1=1$ and
$g_i=0$ for $i \geq 2$.
\item[Case 3:] If $\eta=0$, the system does not get affected by the negative
customers and our model reduces to $GI^X/M/1$ queue with disaster.
Due to the impact of disasters, the system will always remain stable
and hence the c.e. will have exactly $b$ roots inside the unit
circle under the sufficient condition $\delta
>0$. The steady-state distributions can be derived from (\ref{21})-(\ref{23}) after
obtaining the constants from (\ref{18}) and (\ref{19}). Similarly as
before, the solution for $GI/M/1$ queue with disasters (Park et al.
\cite{park2009analysis}) can also be obtained.
\end{description}
\section{Numerical Observation}\label{sec5}
\begin{table}[h!]
\centering \caption{Steady-state distribution of the number of
positive customers in the system at various epochs for different
inter-arrival time distributions} \label{table1} \footnotesize
\begin{tabular}{|c|ccc|ccc|}
\hline ~ & \multicolumn{3}{c}{$GI=M$} & \multicolumn{3}{c|}{$GI=D$}\\
\hline $n$ & $p_n^-$   & $p_n$   &   $p_{n+1}^-/p_n^-$   &   $p_n^-$
& $p_n$     & $p_{n+1}^-/p_n^-$ \\ \hline
0   &   0.20533567  &   0.20533567  &   0.15065676  & 0.23080160  &   0.12004016  &   0.15318913  \\
1   &   0.03093521  &   0.03093521  &   0.91498603  & 0.03535630  &   0.03653976  &   1.12629474  \\
2   &   0.02830528  &   0.02830528  &   1.65427828  & 0.03982161  &   0.03420904  &   1.01859822  \\
3   &   0.04682481  &   0.04682481  &   0.55001705  & 0.04056222  &   0.06060381  &   0.85997718  \\
4   &   0.02575445  &   0.02575445  &   1.23727398  & 0.03488259  &   0.02886916  &   1.20959058  \\
5   &   0.03186531  &   0.03186531  &   1.45536181  & 0.04219365  &   0.04113680  &   0.92958177  \\
6   &   0.04637555  &   0.04637555  &   0.55360053  & 0.03922244  &   0.05948070  &   0.75644328  \\
7   &   0.02567353  &   0.02567353  &   1.05065616  & 0.02966955  &   0.03095702  &   1.07090478  \\
$\vdots$ & $\vdots$ & $\vdots$ &$\vdots$ &$\vdots$ &$\vdots$ & $\vdots$  \\
200 &   0.00000060  &   0.00000060  &   0.94509121  & 0.00000009  &   0.00000010  &   0.93533903  \\
201 &   0.00000057  &   0.00000057  &   0.94509121  & 0.00000008  &   0.00000009  &   0.93533903  \\
202 &   0.00000054  &   0.00000054  &   0.94509121  & 0.00000008  &   0.00000009  &   0.93533903  \\
203 &   0.00000051  &   0.00000051  &   0.94509121  & 0.00000007  &   0.00000008  &   0.93533903  \\
204 &   0.00000048  &   0.00000048  &   0.94509121  & 0.00000007  &   0.00000008  &   0.93533903  \\
205 &   0.00000045  &   0.00000045  &   0.94509121  & 0.00000006  &   0.00000007  &   0.93533903  \\
$\vdots$ & $\vdots$ & $\vdots$ &$\vdots$ &$\vdots$ &$\vdots$ &
$\vdots$  \\ \hline sum &   1.00000000  &   1.00000000  &~ &
1.00000000  &   1.00000000  & \\ \hline
Mean    &   15.04001756 &   15.04001756 & ~ &   12.39890533 &   14.40030123 & ~  \\
\hline
\end{tabular}
\end{table}
%
%
\begin{figure}[htbp]
  \begin{minipage}[b]{0.5\linewidth}
    \centering
    \includegraphics[width=\linewidth]{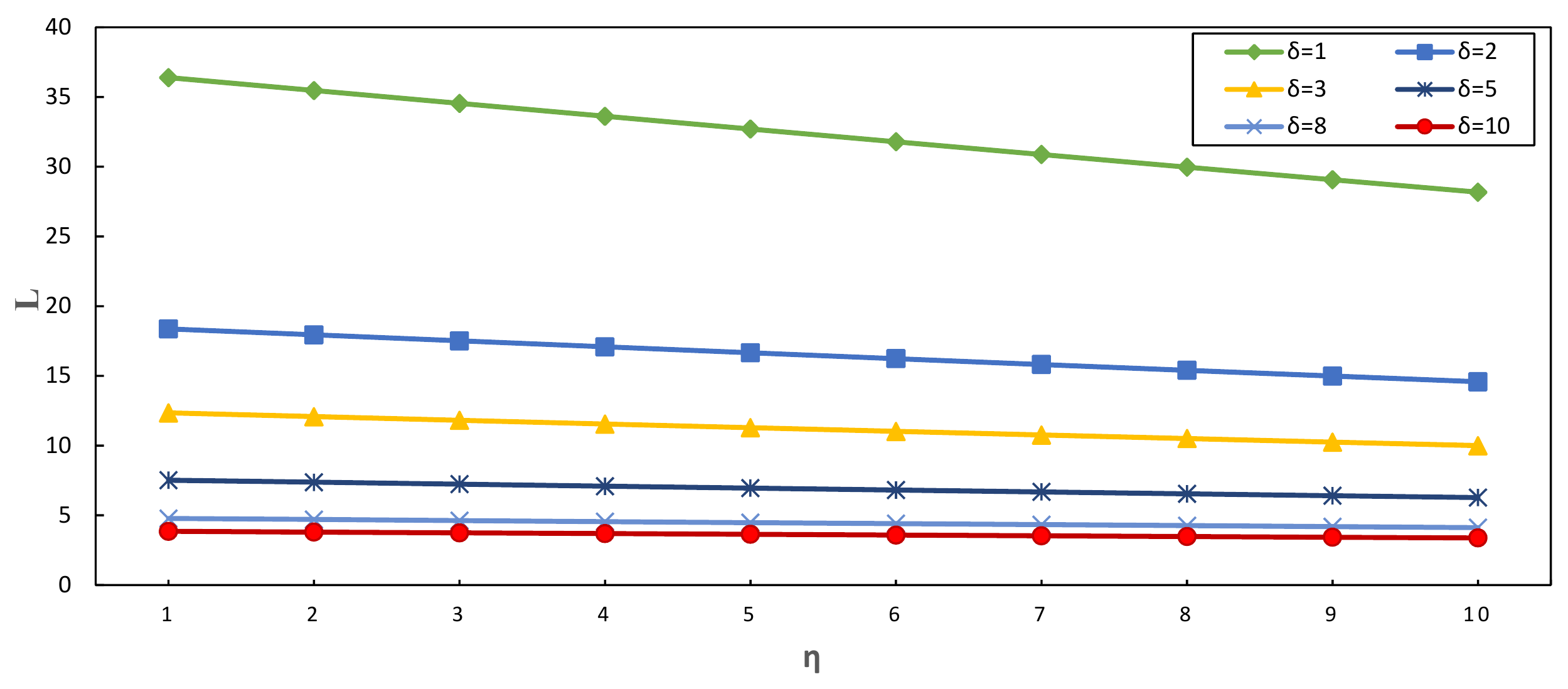}
    \caption{Effect of $\eta$ on $L$ for various $\delta$} \label{fig1}
    \end{minipage}
  \begin{minipage}[b]{0.5\linewidth}
    \centering
    \includegraphics[width=\linewidth]{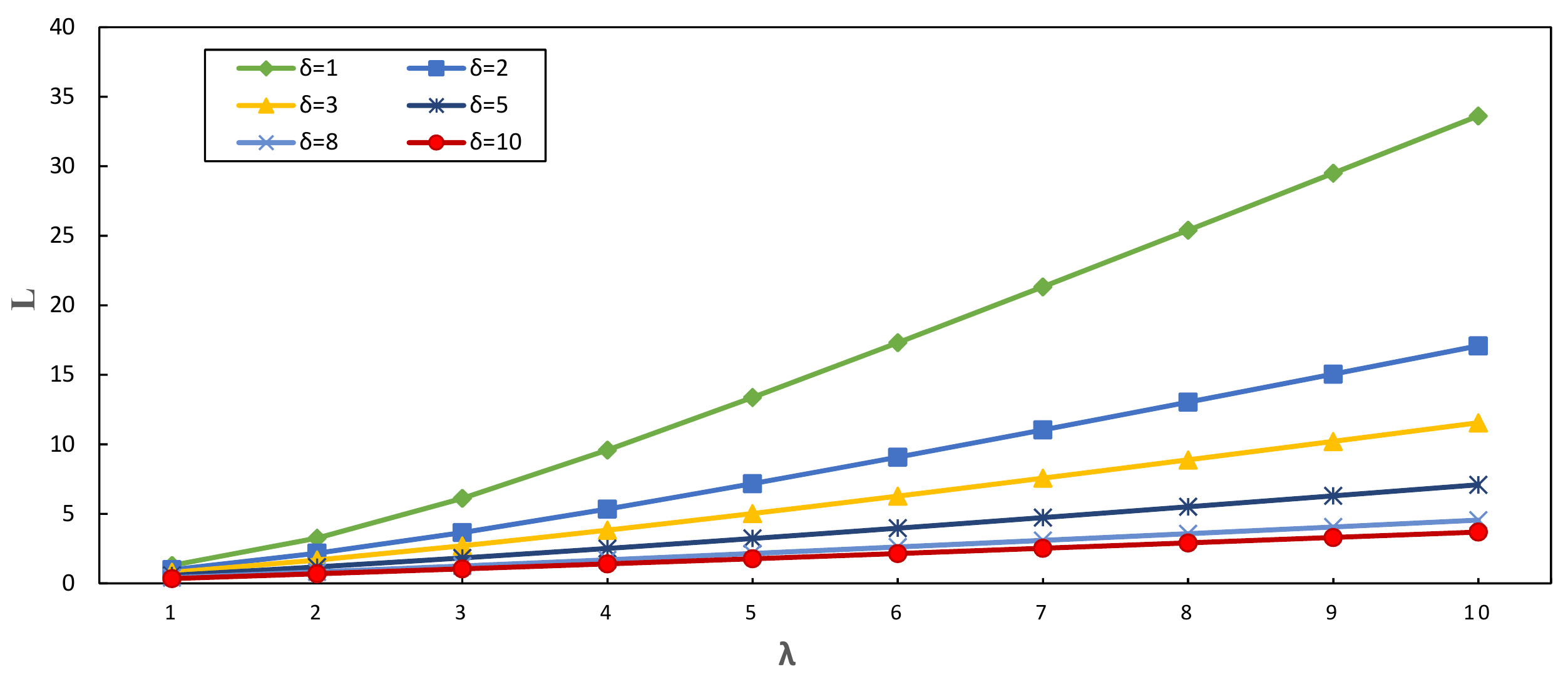}
    \caption{Effect of $\lambda$ on $L$ for various $\delta$} \label{fig2}
     \end{minipage}
\end{figure}
\begin{figure}[htbp]
  \begin{minipage}[b]{0.5\linewidth}
    \centering
    \includegraphics[width=\linewidth]{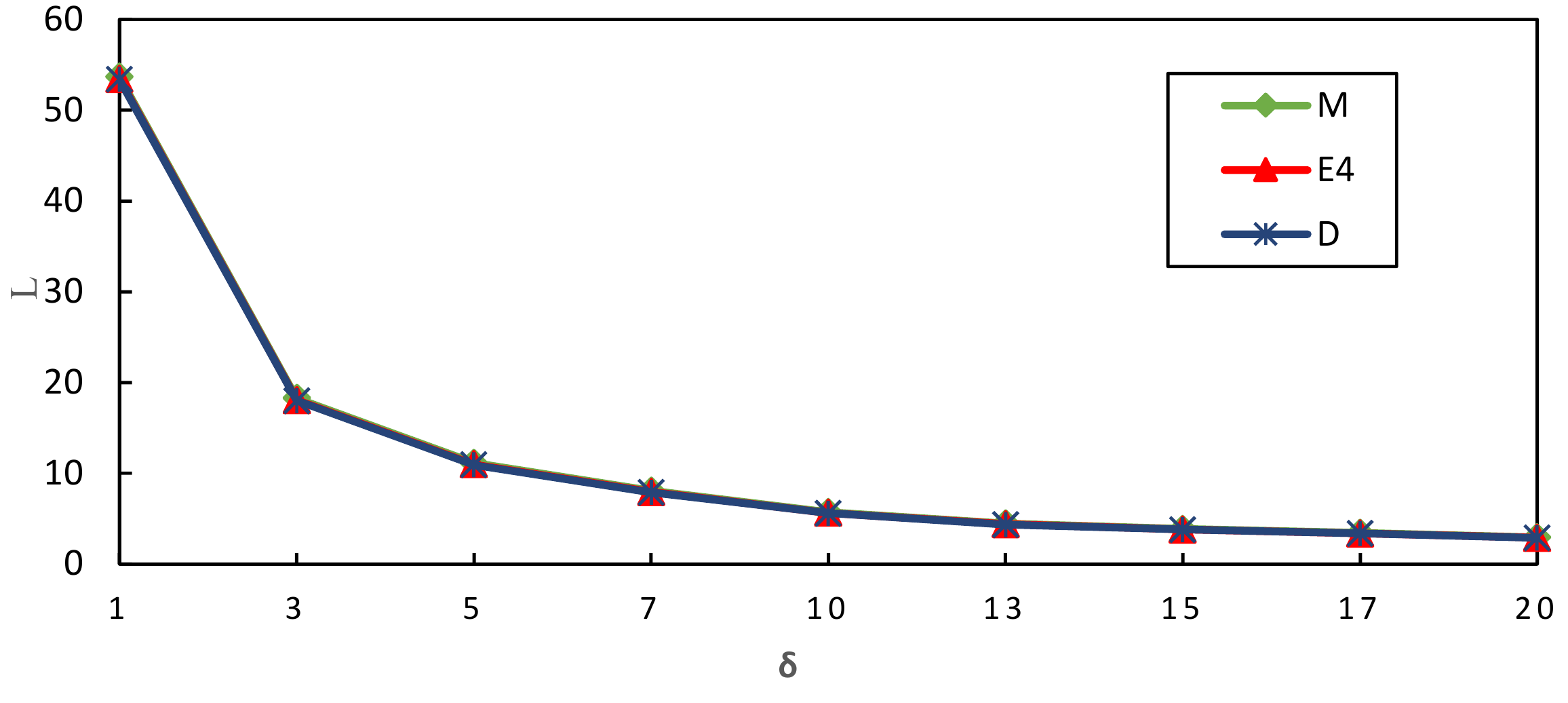}
    \caption{Effect of $\delta$ on $L$ for various interarrival distributions} \label{fig3}
    \end{minipage}
  \hspace{0.15cm}
  \begin{minipage}[b]{0.5\linewidth}
    \centering
    \includegraphics[width=\linewidth]{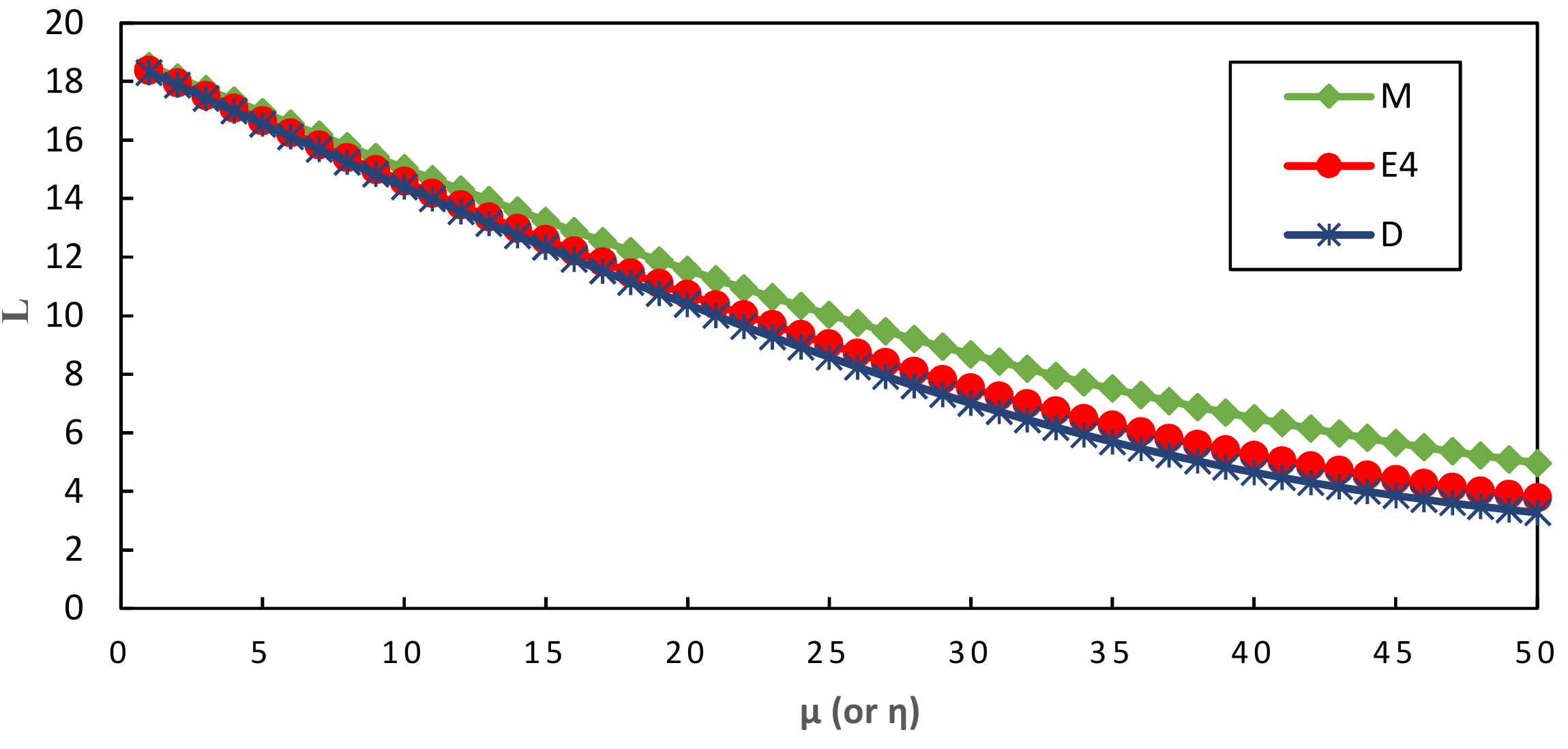}
    \caption{Effect of $\eta$ on $L$ for various interarrival distributions} \label{fig4}
     \end{minipage}
\end{figure}
In this section we demonstrate the analytical results obtained in
Section \ref{sec3} by some numerical examples, which are represented
in tabular and graphical form. The results given in the table may be
beneficial for other researchers who would like to compare their
results using some other methods in the near future.
\par Table \ref{table1} displays the steady-state distribution of
the number of positive customers in the system for Poisson ($M$) and
deterministic ($D$) arrival processes. The parameters chosen are
$\lambda=10$, $\mu=10$, $\eta=5$, $\delta=2$, $g_1=0.2$, $g_3=0.4$,
$g_6=0.3$ and $g_{10}=0.1$. The last row of the table depicts the
average system content at various epochs. It is important to note
that the system-content distributions in the $2^{nd}$ and $3^{rd}$
column are same due to the Poisson arrival process, which verifies
the accuracy of our analytical results. Meanwhile, for deterministic
inter-arrival time distribution, the L.S.T $A^*(s)$ is a
transcendental function, which is approximated to a rational
function using $Pad\acute{e}(15,15)$ approximation (see Akar and
Arikan \cite{AkarArikan}, Singh \textit{et al.} \cite{gagan}).
Another interesting trend can be observed in the $4^{th}$ and
$7^{th}$ column of the table. As $n$ becomes larger, the ratio of
the system content distribution at pre-arrival epoch converges to a
particular value which is the largest real root (say $r_b$) of the
c.e. (\ref{12}) lying inside the unit circle. This suggests that the
limiting distributions at the pre-arrival epoch can be approximated
by the unique largest root of the c.e. as
$p_n^-=\frac{1}{\lambda}c_br_b^n$.
\par Figure \ref{fig1} investigates the influence of $\eta$ on $L$
for different values of $\delta$. As $\eta$ increase, $L$ decreases
for any value of $\delta$, which is intuitive. Similarly, for a
fixed $\eta$, $L$ decreases with increasing $\delta$. However, as
$\delta$ becomes too large ($\delta =10$), $L$ seems to attain a
constant value irrespective of the values of $\eta$. A similar
behavior can be experienced on plotting $L$ against $\mu$ for
different values of $\delta$, and consequently it is omitted. Figure
\ref{fig2} depicts the impact of $\lambda$ on $L$ for different
$\delta$. Clearly, as $\lambda$ increases $L$ increases for any
$\delta$. However, when $\lambda$ is kept fixed along with other
parameters, $L$ decreases significantly with the increase in
$\delta$.
\par Finally, in Figure \ref{fig3} and \ref{fig4} we respectively
illustrate the impact of $\delta$ and $\eta$ on $L$ for different
inter-arrival time distributions, namely, exponential ($M$), Erlang
($E_4$) and deterministic ($D$). It may be observed in Figure
\ref{fig3} that for each inter-arrival time distribution, $L$
decreases as $\delta$ increases, which is obvious. However, for a
fixed $\delta$, $L$ is equal for all the three distributions. A
possible explanation for this phenomenon may be the frequent
occurrence of disasters which removes all the customers including
the batch which has just arrived. The effect of inter-arrival time
distribution can be best understood from Figure \ref{fig4} as $\eta$
increases. For higher values of $\eta$, $L$ decreases significantly.
However, for exponential inter-arrival time distribution $L$ is
greater, and decreases for Erlang followed by deterministic
distribution. It may be mentioned that in all the numerical results
generated throughout this section, the values of the parameters
involved are not restricted to any condition except that $\delta
>0$, as the system with disaster is always stable.
\section{Concluding remarks}\label{sec6}
In this paper, the steady-state analysis of a $GI^X/M/1$ queueing
model with negative customers and disasters has been presented. We
have derived the explicit closed-form expressions of the
distribution of the number of positive customers in the system at
pre-arrival and arbitrary epochs, in terms of roots of the
associated characteristic equation. The results of some classical
queueing models with or without negative arrivals have been
discussed along with their stability conditions. Additionally,
through some numerical examples, we have investigated the influence
of negative customers and disasters on the performance
characteristic of the system. The methodology used in this paper is
based on supplementary variable technique and difference equation
method which makes the analysis easily tractable, both theoretically
and computationally. The procedure developed throughout the analysis
can be utilized and further extended to study some more complicated
models.
\section{Appendix}
\begin{theorem}
The c.e. $A^*( \delta + (\mu+\eta) (1-z)) \sum_{i=1}^{b} g_i z^{b-i}
- z^b = 0$ have exactly $b$ roots inside the unit circle $|z|=1$
subject to the condition $\delta >0$.
\end{theorem}
\begin{proof}
Let us assume $f_1(z)=-z^b$ and $f_2(z)=A^*( \delta + (\mu+\eta)
(1-z)) \sum_{i=1}^{b} g_i z^{b-i}= K(z)$. Since $A^*( \delta +
(\mu+\eta) (1-z)) \sum_{i=1}^{b} g_i z^{b-i}$ is an analytic
function, it can be written in the form $K(z)=
\sum_{i=1}^{\infty}k_iz^i$ such that $k_i \geq 0$ for all $i$.
Consider the circle $|z|=1- \epsilon$ where $\epsilon
>0$ and is a sufficiently small quantity. Now
\begin{eqnarray}
|f_1(z)|&=&|z^b|=(1-\epsilon)^b=1-b\epsilon +o(\epsilon) \nonumber \\
|f_2(z)|&=& |K(z)||\sum_{i=1}^{b} g_i z^{b-i}|\leq
K(|z|)\sum_{i=1}^{b} g_i|z|^{b-i}
=K(1-\epsilon)\sum_{i=1}^bg_i(1-\epsilon)^{b-i} \nonumber \\ 
&=& A^*( \delta ) - \epsilon \{2A^*( \delta ) (b-\overline{g}) -
(\mu+\eta) A^{*(1)}( \delta ) \} + o(\epsilon) \nonumber \\ &<& 1-
b\epsilon +o(\epsilon) \nonumber
\end{eqnarray}
under the sufficient condition $\delta >0$. Thus from
Rouch$\acute{e}$'s theorem we have exactly the same number of zeroes
in $f_1(z)$ and $f_1(z)+f_2(z)$ inside the unit circle, and hence
the theorem.
\end{proof}


\end{document}